\documentclass[11pt,fleqn]{article}

\usepackage{amsmath,amssymb,amsthm}
\usepackage{geometry}
\usepackage[pdfpagemode=UseNone,pdfstartview=FitH,hypertex]{hyperref}

\newcommand{\A}{{\cal A}}
\newcommand{\abs}[1]{\left|#1\right|}
\newcommand{\bdry}[1]{\partial #1}
\newcommand{\closure}[1]{\overline{#1}}
\newcommand{\dint}{\ds{\int}}
\newcommand{\dist}[2]{\text{dist}\, (#1,#2)}
\newcommand{\ds}[1]{\displaystyle #1}

\newcommand{\eps}{\varepsilon}
\newcommand{\F}{{\cal F}}

\newcommand{\hquad}{\hspace{0.08in}}
\newcommand{\id}[1]{id_{#1}}
\newcommand{\incl}{\subset}
\newcommand{\N}{\mathbb N}
\newcommand{\norm}[2][]{\left\|#2\right\|_{#1}}
\renewcommand{\O}{\text{O}}
\renewcommand{\o}{\text{o}}
\newcommand{\PS}[1]{$(\text{PS})_{#1}$}
\newcommand{\pnorm}[2][]{\if #1'' \left|#2\right|_p \else \left|#2\right|_{#1} \fi}
\newcommand{\R}{\mathbb R}
\newcommand{\RP}{\R \text{P}}
\newcommand{\restr}[2]{\left.#1\right|_{#2}}
\newcommand{\seq}[1]{\left(#1\right)}
\newcommand{\set}[1]{\left\{#1\right\}}

\newcommand{\Z}{\mathbb Z}

\newenvironment{enumroman}{\begin{enumerate}
\renewcommand{\theenumi}{$(\roman{enumi})$}
\renewcommand{\labelenumi}{$(\roman{enumi})$}}{\end{enumerate}}

\newtheorem{corollary}{Corollary}[section]
\newtheorem{lemma}[corollary]{Lemma}
\newtheorem{proposition}[corollary]{Proposition}
\newtheorem{theorem}[corollary]{Theorem}

\theoremstyle{remark}

\numberwithin{equation}{section}

\title{\bf On a class of critical $N$-Laplacian problems}
\author{\bf Tsz Chung Ho and Kanishka Perera\\
Department of Mathematical Sciences\\
Florida Institute of Technology\\
Melbourne, FL 32901, USA\\
\em tho2011@my.fit.edu \& kperera@fit.edu}
\date{}

\begin{document}

\maketitle

\begin{abstract}
We establish some existence results for a class of critical $N$-Laplacian problems in a bounded domain in $\R^N$. In the absence of a suitable direct sum decomposition of the underlying Sobolev space to which the classical linking theorem can be applied, we use an abstract linking theorem based on the $\Z_2$-cohomological index to obtain a nontrivial critical point.
\end{abstract}

\section{Introduction}

In this paper we establish some existence results for the class of critical $N$-Laplacian problems
\begin{equation} \label{1.1}
\left\{\begin{aligned}
- \Delta_N\, u & = h(u)\, e^{\alpha\, |u|^{N'}} && \text{in } \Omega\\[5pt]
u & = 0 && \text{on } \bdry{\Omega},
\end{aligned}\right.
\end{equation}
where $\Omega$ is a smooth bounded domain in $\R^N,\, N \ge 2$, $\alpha > 0$, $N' = N/(N - 1)$ is the H\"{o}lder conjugate of $N$, and $h$ is a continuous function such that
\begin{equation} \label{1.1.1}
\lim_{|t| \to \infty}\, h(t) = 0
\end{equation}
and
\begin{equation} \label{1.2}
0 < \beta := \liminf_{|t| \to \infty}\, th(t) < \infty.
\end{equation}
This problem is motivated by the Trudinger-Moser inequality
\begin{equation} \label{1.3}
\sup_{\substack{u \in W^{1,N}_0(\Omega)\\ \norm{u} \le 1}}\, \int_\Omega e^{\alpha_N\, |u|^{N'}} dx < \infty,
\end{equation}
where $W^{1,N}_0(\Omega)$ is the usual Sobolev space with the norm
\begin{gather*}
\norm{u} = \left(\int_\Omega |\nabla u|^N dx\right)^{1/N},\\[7.5pt]
\alpha_N = N\, \omega_{N-1}^{1/(N-1)},
\end{gather*}
and
\[
\omega_{N-1} = \frac{2 \pi^{N/2}}{\Gamma(N/2)}
\]
is the area of the unit sphere in $\R^N$ (see Trudinger \cite{MR0216286} and Moser \cite{MR0301504}). Problem \eqref{1.1} is critical with respect to this inequality and hence lacks compactness. Indeed, the associated variational functional satisfies the Palais-Smale compactness condition only at energy levels below a certain threshold (see Proposition \ref{Proposition 2.1} in the next section).

In dimension $N = 2$, problem \eqref{1.1} is semilinear and has been extensively studied in the literature (see, e.g., \cite{MR1044289,MR1386960,MR1399846,MR2772124}). In dimensions $N \ge 3$, this problem is quasilinear and has been studied mainly when
\begin{equation} \label{1.20}
G(t) := \int_0^t h(s)\, e^{\alpha\, |s|^{N'}} ds \le \lambda\, |t|^N \quad \text{for small } t
\end{equation}
for some $\lambda \in (0,\lambda_1)$ (see, e.g., \cite{MR1079983,MR1865413,MR1392090}). Here
\begin{equation} \label{1.5}
\lambda_1 = \inf_{u \in W^{1,N}_0(\Omega) \setminus \set{0}}\, \frac{\dint_\Omega |\nabla u|^N dx}{\dint_\Omega |u|^N dx}
\end{equation}
is the first eigenvalue of the eigenvalue problem
\begin{equation} \label{1.4}
\left\{\begin{aligned}
- \Delta_N\, u & = \lambda\, |u|^{N-2}\, u && \text{in } \Omega\\[5pt]
u & = 0 && \text{on } \bdry{\Omega}.
\end{aligned}\right.
\end{equation}
The case $h(t) = \lambda\, |t|^{N-2}\, t$ with $\lambda > 0$, for which $\beta = \infty$, was recently studied in Yang and Perera \cite{MR3616328}. The remaining case, where $N \ge 3$, $\lambda \ge \lambda_1$, and $\beta < \infty$, does not seem to have been studied in the literature. This case is covered in our results here, which are for large $\beta < \infty$ and allow $N \ge 3$ and $\lambda \ge \lambda_1$ in \eqref{1.20}.

Let $d$ be the radius of the largest open ball contained in $\Omega$. Our first result is the following theorem.

\begin{theorem} \label{Theorem 1.1}
Assume that $\alpha > 0$, $h$ satisfies \eqref{1.1.1} and \eqref{1.2}, and $G$ satisfies
\begin{gather}
\label{1.9} G(t) \ge - \frac{1}{N}\, \sigma_0\, |t|^N \quad \text{for } t \ge 0,\\[7.5pt]
\label{1.7} G(t) \le \frac{1}{N}\, (\lambda_1 - \sigma_1)\, |t|^N \quad \text{for } |t| \le \delta
\end{gather}
for some $\sigma_0 \ge 0$ and $\sigma_1, \delta > 0$. If
\begin{equation} \label{1.10}
\beta > \frac{1}{N \alpha^{N-1}} \left(\frac{N}{d}\right)^N\! e^{\sigma_0/(N-1)\, \kappa},
\end{equation}
where $\kappa = \dfrac{1}{N!} \left(\dfrac{N}{d}\right)^N$, then problem \eqref{1.1} has a nontrivial solution.
\end{theorem}

In particular, we have the following corollary for $\sigma_0 = 0$.

\begin{corollary} \label{Corollary 1.2}
Assume that $\alpha > 0$, $h$ satisfies \eqref{1.1.1} and \eqref{1.2}, and $G$ satisfies
\begin{gather*}
G(t) \ge 0 \quad \text{for } t \ge 0,\\[7.5pt]
G(t) \le \frac{1}{N}\, (\lambda_1 - \sigma_1)\, |t|^N \quad \text{for } |t| \le \delta
\end{gather*}
for some $\sigma_1, \delta > 0$. If
\[
\beta > \frac{1}{N \alpha^{N-1}} \left(\frac{N}{d}\right)^N,
\]
then problem \eqref{1.1} has a nontrivial solution.
\end{corollary}

Corollary \ref{Corollary 1.2} should be compared with Theorem 1 of do {\'O} \cite{MR1392090}, where this result is proved under the stronger assumption $h(t) \ge 0$ for $t \ge 0$.

To state our second result, let $\seq{\lambda_k}$ be the sequence of eigenvalues of problem \eqref{1.4} based on the $\Z_2$-cohomological index that was introduced in Perera \cite{MR1998432} (see Proposition \ref{Proposition 2.3} in the next section). We have the following theorem.

\begin{theorem} \label{Theorem 1.3}
Assume that $\alpha > 0$, $h$ satisfies \eqref{1.1.1} and \eqref{1.2}, and $G$ satisfies
\begin{gather}
\label{1.11} G(t) \ge \frac{1}{N}\, (\lambda_{k-1} + \sigma_0)\, |t|^N \quad \forall t,\\[7.5pt]
\label{1.12} G(t) \le \frac{1}{N}\, (\lambda_k - \sigma_1)\, |t|^N \quad \text{for } |t| \le \delta
\end{gather}
for some $k \ge 2$ and $\sigma_0, \sigma_1, \delta > 0$. Then there exists a constant $c > 0$ depending on $\Omega$, $\alpha$, and $k$, but not on $\sigma_0$, $\sigma_1$, or $\delta$, such that if
\[
\beta > \frac{1}{\alpha^{N-1}} \left(\frac{N}{d}\right)^N\! e^{c/\sigma_0^{N-1}},
\]
then problem \eqref{1.1} has a nontrivial solution.
\end{theorem}

Theorem \ref{Theorem 1.3} should be compared with Theorem 1.4 of de Figueiredo et al.\! \cite{MR1386960, MR1399846}, where this result is proved in the case $N = 2$ under the additional assumption that $0 < 2G(t) \le th(t)\, e^{\alpha t^2}$ for all $t \in \R \setminus \set{0}$. However, the linking argument used in \cite{MR1386960, MR1399846} is based on a splitting of $H^1_0(\Omega)$ that involves the eigenspaces of the Laplacian, and this argument does not extend to the case $N \ge 3$ where the $N$-Laplacian is a nonlinear operator and therefore has no linear eigenspaces. We will prove Theorem \ref{Theorem 1.3} using an abstract critical point theorem based on the $\Z_2$-cohomological index that was proved in Yang and Perera \cite{MR3616328} (see Section \ref{Subsection 2.4}).

In the proofs of Theorems \ref{Theorem 1.1} and \ref{Theorem 1.3}, the inner radius $d$ of $\Omega$ comes into play when verifying that certain minimax levels are below the compactness threshold given in Proposition \ref{Proposition 2.1}.

\section{Preliminaries}

\subsection{A compactness result}

Weak solutions of problem \eqref{1.1} coincide with critical points of the $C^1$-functional
\[
E(u) = \frac{1}{N} \int_\Omega |\nabla u|^N dx - \int_\Omega G(u)\, dx, \quad u \in W^{1,N}_0(\Omega).
\]
We recall that a \PS{c} sequence of $E$ is a sequence $\seq{u_j} \subset W^{1,N}_0(\Omega)$ such that $E(u_j) \to c$ and $E'(u_j) \to 0$. Proofs of Theorem \ref{Theorem 1.1} and Theorem \ref{Theorem 1.3} will be based on the following compactness result.

\begin{proposition} \label{Proposition 2.1}
Assume that $\alpha > 0$ and $h$ satisfies \eqref{1.1.1} and \eqref{1.2}. Then for all $c \ne 0$ satisfying
\[
c < \frac{1}{N} \left(\frac{\alpha_N}{\alpha}\right)^{N-1},
\]
every {\em \PS{c}} sequence of $E$ has a subsequence that converges weakly to a nontrivial solution of problem \eqref{1.1}.
\end{proposition}

\begin{proof}
Let $\seq{u_j} \subset W^{1,N}_0(\Omega)$ be a \PS{c} sequence of $E$. Then
\begin{equation} \label{2.1}
E(u_j) = \frac{1}{N} \norm{u_j}^N - \int_\Omega G(u_j)\, dx = c + \o(1)
\end{equation}
and
\begin{equation} \label{2.2}
E'(u_j)\, u_j = \norm{u_j}^N - \int_\Omega u_j\, h(u_j)\, e^{\alpha\, |u_j|^{N'}} dx = \o(\norm{u_j}).
\end{equation}
First we show that $\seq{u_j}$ is bounded in $W^{1,N}_0(\Omega)$. Multiplying \eqref{2.1} by $2N$ and subtracting \eqref{2.2} gives
\[
\norm{u_j}^N + \int_\Omega \left(u_j\, h(u_j)\, e^{\alpha\, |u_j|^{N'}} - 2N G(u_j)\right) dx = 2Nc + \o(\norm{u_j} + 1),
\]
so it suffices to show that $th(t)\, e^{\alpha\, |t|^{N'}} - 2N G(t)$ is bounded from below. Let $0 < \eps < \beta/(2N + 1)$. By \eqref{1.1.1} and \eqref{1.2}, for some constant $C_\eps > 0$,
\begin{equation} \label{2.5}
|G(t)| \le \eps\, e^{\alpha\, |t|^{N'}} + C_\eps
\end{equation}
and
\begin{equation} \label{2.3}
th(t)\, e^{\alpha\, |t|^{N'}} \ge (\beta - \eps)\, e^{\alpha\, |t|^{N'}} - C_\eps
\end{equation}
for all $t$. So
\[
th(t)\, e^{\alpha\, |t|^{N'}} - 2N G(t) \ge [\beta - (2N + 1)\, \eps]\, e^{\alpha\, |t|^{N'}} - (2N + 1)\, C_\eps,
\]
which is bounded from below.

Since $\seq{u_j}$ is bounded in $W^{1,N}_0(\Omega)$, a renamed subsequence converges to some $u$ weakly in $W^{1,N}_0(\Omega)$, strongly in $L^p(\Omega)$ for all $p \in [1,\infty)$, and a.e.\! in $\Omega$. We have
\begin{equation} \label{2.6}
E'(u_j)\, v = \int_\Omega |\nabla u_j|^{N-2}\, \nabla u_j \cdot \nabla v\, dx - \int_\Omega v\, h(u_j)\, e^{\alpha\, |u_j|^{N'}} dx \to 0
\end{equation}
for all $v \in W^{1,N}_0(\Omega)$. By \eqref{1.1.1}, given any $\eps > 0$, there exists a constant $C_\eps > 0$ such that
\begin{equation} \label{2.7}
|h(t)\, e^{\alpha\, |t|^{N'}}| \le \eps\, e^{\alpha\, |t|^{N'}} + C_\eps \quad \forall t.
\end{equation}
By \eqref{2.2},
\[
\sup_j\, \int_\Omega u_j\, h(u_j)\, e^{\alpha\, |u_j|^{N'}} dx < \infty,
\]
which together with \eqref{2.3} gives
\begin{equation} \label{2.8}
\sup_j\, \int_\Omega e^{\alpha\, |u_j|^{N'}} dx < \infty.
\end{equation}
For $v \in C^\infty_0(\Omega)$, it follows from \eqref{2.7} and \eqref{2.8} that the sequence $(v\, h(u_j)\, e^{\alpha\, |u_j|^{N'}})$ is uniformly integrable and hence
\[
\int_\Omega v\, h(u_j)\, e^{\alpha\, |u_j|^{N'}} dx \to \int_\Omega v\, h(u)\, e^{\alpha\, |u|^{N'}} dx
\]
by Vitali's convergence theorem, so it follows from \eqref{2.6} that
\[
\int_\Omega |\nabla u|^{N-2}\, \nabla u \cdot \nabla v\, dx - \int_\Omega v\, h(u)\, e^{\alpha\, |u|^{N'}} dx = 0.
\]
Then this holds for all $v \in W^{1,N}_0(\Omega)$ by density, so the weak limit $u$ is a solution of problem \eqref{1.1}.

Suppose that $u = 0$. Then
\[
\int_\Omega G(u_j)\, dx \to 0
\]
since \eqref{2.5} and \eqref{2.8} imply that the sequence $\seq{G(u_j)}$ is uniformly integrable, so \eqref{2.1} gives $c \ge 0$ and
\begin{equation} \label{2.9}
\norm{u_j} \to (Nc)^{1/N}.
\end{equation}
Let $Nc < \nu < (\alpha_N/\alpha)^{N-1}$. Then $\norm{u_j} \le \nu^{1/N}$ for all $j \ge j_0$ for some $j_0$. Let $q = \alpha_N/\alpha \nu^{1/(N-1)} > 1$. By the H\"{o}lder inequality,
\[
\abs{\int_\Omega u_j\, h(u_j)\, e^{\alpha\, |u_j|^{N'}} dx} \le \left(\int_\Omega |u_j\, h(u_j)|^p\, dx\right)^{1/p}\! \left(\int_\Omega e^{q \alpha\, |u_j|^{N'}} dx\right)^{1/q},
\]
where $1/p + 1/q = 1$. The first integral on the right-hand side converges to zero since $h$ is bounded and $u_j \to 0$ in $L^p(\Omega)$, and the second integral is bounded by \eqref{1.3} since $q \alpha\, |u_j|^{N'} = \alpha_N\, |\widetilde{u}_j|^{N'}$, where $\widetilde{u}_j = u_j/\nu^{1/N}$ satisfies $\norm{\widetilde{u}_j} \le 1$ for $j \ge j_0$, so
\[
\int_\Omega u_j\, h(u_j)\, e^{\alpha\, |u_j|^{N'}} dx \to 0.
\]
Then $u_j \to 0$ by \eqref{2.2} and hence $c = 0$ by \eqref{2.9}, contrary to assumption. So $u$ is a nontrivial solution.
\end{proof}

\subsection{$\Z_2$-cohomological index}

The $\Z_2$-cohomological index of Fadell and Rabinowitz \cite{MR0478189} is defined as follows. Let $W$ be a Banach space and let $\A$ denote the class of symmetric subsets of $W \setminus \set{0}$. For $A \in \A$, let $\overline{A} = A/\Z_2$ be the quotient space of $A$ with each $u$ and $-u$ identified, let $f : \overline{A} \to \RP^\infty$ be the classifying map of $\overline{A}$, and let $f^\ast : H^\ast(\RP^\infty) \to H^\ast(\overline{A})$ be the induced homomorphism of the Alexander-Spanier cohomology rings. The cohomological index of $A$ is defined by
\[
i(A) = \begin{cases}
\sup \set{m \ge 1 : f^\ast(\omega^{m-1}) \ne 0}, & A \ne \emptyset\\[5pt]
0, & A = \emptyset,
\end{cases}
\]
where $\omega \in H^1(\RP^\infty)$ is the generator of the polynomial ring $H^\ast(\RP^\infty) = \Z_2[\omega]$. For example, the classifying map of the unit sphere $S^{m-1}$ in $\R^m,\, m \ge 1$ is the inclusion $\RP^{m-1} \incl \RP^\infty$, which induces isomorphisms on $H^q$ for $q \le m - 1$, so $i(S^{m-1}) = m$.

The following proposition summarizes the basic properties of the cohomological index (see Fadell and Rabinowitz \cite{MR0478189}).

\begin{proposition}
The index $i : \A \to \N \cup \set{0,\infty}$ has the following properties:
\begin{enumroman}
\item Definiteness: $i(A) = 0$ if and only if $A = \emptyset$.
\item Monotonicity: If there is an odd continuous map from $A$ to $B$ (in particular, if $A \subset B$), then $i(A) \le i(B)$. Thus, equality holds when the map is an odd homeomorphism.
\item Dimension: $i(A) \le \dim W$.
\item Continuity: If $A$ is closed, then there is a closed neighborhood $N \in \A$ of $A$ such that $i(N) = i(A)$. When $A$ is compact, $N$ may be chosen to be a $\delta$-neighborhood $N_\delta(A) = \set{u \in W : \dist{u}{A} \le \delta}$.
\item Subadditivity: If $A$ and $B$ are closed, then $i(A \cup B) \le i(A) + i(B)$.
\item Stability: If $SA$ is the suspension of $A \ne \emptyset$, obtained as the quotient space of $A \times [-1,1]$ with $A \times \set{1}$ and $A \times \set{-1}$ collapsed to different points, then $i(SA) = i(A) + 1$.
\item Piercing property: If $A$, $A_0$ and $A_1$ are closed, and $\varphi : A \times [0,1] \to A_0 \cup A_1$ is a continuous map such that $\varphi(-u,t) = - \varphi(u,t)$ for all $(u,t) \in A \times [0,1]$, $\varphi(A \times [0,1])$ is closed, $\varphi(A \times \set{0}) \subset A_0$ and $\varphi(A \times \set{1}) \subset A_1$, then $i(\varphi(A \times [0,1]) \cap A_0 \cap A_1) \ge i(A)$.
\item Neighborhood of zero: If $U$ is a bounded closed symmetric neighborhood of $0$, then $i(\bdry{U}) = \dim W$.
\end{enumroman}
\end{proposition}

\subsection{Eigenvalues} \label{Subsection 2.3}

Eigenvalues of problem \eqref{1.4} coincide with critical values of the functional
\[
\Psi(u) = \frac{1}{\dint_\Omega |u|^N dx}, \quad u \in S = \set{u \in W^{1,N}_0(\Omega) : \int_\Omega |\nabla u|^N dx = 1}.
\]
We have the following proposition (see Perera \cite{MR1998432} and Perera et al.\! \cite[Proposition 3.52 and Proposition 3.53]{MR2640827}).

\begin{proposition} \label{Proposition 2.3}
Let $\F$ denote the class of symmetric subsets of $S$ and set
\[
\lambda_k := \inf_{\substack{M \in \F\\ i(M) \ge k}}\, \sup_{u \in M}\, \Psi(u), \quad k \in \N.
\]
Then $0 < \lambda_1 < \lambda_2 \le \lambda_3 \le \cdots \to + \infty$ is a sequence of eigenvalues of problem \eqref{1.4}. Moreover, if $\lambda_{k-1} < \lambda_k$, then
\[
i(\Psi^{\lambda_{k-1}}) = i(S \setminus \Psi_{\lambda_k}) = k - 1,
\]
where $\Psi^a = \set{u \in S : \Psi(u) \le a}$ and $\Psi_a = \set{u \in S : \Psi(u) \ge a}$ for $a \in \R$.
\end{proposition}

We will also need the following result of Degiovanni and Lancelotti (\cite[Theorem 2.3]{MR2514055}).

\begin{proposition} \label{Proposition 2.4}
If $\lambda_{k-1} < \lambda_k$, then $\Psi^{\lambda_{k-1}}$ contains a compact symmetric set $C$ of index $k - 1$ that is bounded in $C^1(\closure{\Omega})$.
\end{proposition}

\subsection{An abstract critical point theorem} \label{Subsection 2.4}

We will use the following abstract critical point theorem proved in Yang and Perera \cite[Theorem 2.2]{MR3616328} to prove Theorem \ref{Theorem 1.3}. This result generalizes the linking theorem of Rabinowitz \cite{MR0488128}.

\begin{theorem} \label{Theorem 2.2}
Let $E$ be a $C^1$-functional defined on a Banach space $W$ and let $A_0$ and $B_0$ be disjoint nonempty closed symmetric subsets of the unit sphere $S = \set{u \in W : \norm{u} = 1}$ such that
\begin{equation} \label{2.9.1}
i(A_0) = i(S \setminus B_0) < \infty.
\end{equation}
Assume that there exist $R > \rho > 0$ and $\omega \in S \setminus A_0$ such that
\[
\sup E(A) \le \inf E(B), \qquad \sup E(X) < \infty,
\]
where
\begin{gather*}
A = \set{sv : v \in A_0,\, 0 \le s \le R} \cup \set{R\, \pi((1 - t)\, v + t \omega) : v \in A_0,\, 0 \le t \le 1},\\[5pt]
B = \set{\rho u : u \in B_0},\\[5pt]
X = \set{sv + t \omega : v \in A_0,\, s, t \ge 0,\, \norm{sv + t \omega} \le R},
\end{gather*}
and $\pi : W \setminus \set{0} \to S,\, u \mapsto u/\norm{u}$ is the radial projection onto $S$. Let
\[
\Gamma = \set{\gamma \in C(X,W) : \gamma(X) \text{ is closed and} \restr{\gamma}{A} = \id{A}},
\]
and set
\[
c := \inf_{\gamma \in \Gamma}\, \sup_{u \in \gamma(X)}\, E(u).
\]
Then $\inf E(B) \le c \le \sup E(X)$, and $E$ has a {\em \PS{c}} sequence.
\end{theorem}

\subsection{Moser sequence}

For $j \ge 2$, let
\begin{equation} \label{2.5.1}
\omega_j(x) = \frac{1}{\omega_{N-1}^{1/N}}\, \begin{cases}
(\log j)^{(N-1)/N}, & |x| \le d/j\\[15pt]
\dfrac{\log\, (d/|x|)}{(\log j)^{1/N}}, & d/j < |x| < d\\[15pt]
0, & |x| \ge d
\end{cases}
\end{equation}
(see Moser \cite{MR0301504}).

\begin{proposition} \label{Proposition 2.6}
We have
\begin{equation} \label{3.4.1}
\int_\Omega \omega_j^m\, dx = \frac{m!\, \omega_{N-1}^{1-m/N} d^N}{N^{m+1}\, (\log j)^{m/N}} \left[1 - \frac{1}{j^N} \sum_{l=1}^m \frac{(N \log j)^{m-l}}{(m - l)!}\right], \quad m = 1,\dots,N
\end{equation}
and
\begin{equation} \label{3.4.2}
\int_\Omega |\nabla \omega_j|^m\, dx = \begin{cases}
\dfrac{\omega_{N-1}^{1-m/N} d^{N-m}}{(N - m)\, (\log j)^{m/N}} \left(1 - \dfrac{1}{j^{N-m}}\right), & m = 1,\dots,N - 1\\[15pt]
1, & m = N.
\end{cases}
\end{equation}
\end{proposition}

\begin{proof}
We have
\[
\int_\Omega \omega_j^m\, dx = \frac{\omega_{N-1}^{1-m/N} d^N}{(\log j)^{m/N}} \left[I_m + \frac{(\log j)^m}{N j^N}\right],
\]
where
\[
I_m = \int_{1/j}^1 (- \log s)^m\, s^{N-1}\, ds.
\]
We have
\[
I_1 = \frac{1}{N^2} \left[1 - \frac{1}{j^N}\, (N \log j + 1)\right],
\]
and integrating by parts gives the recurrence relation
\[
I_m = \frac{m}{N}\, I_{m-1} - \frac{(\log j)^m}{N j^N}, \quad m \ge 2.
\]
So
\[
I_m = \frac{m!}{N^{m+1}} \left[1 - \frac{1}{j^N} \sum_{l=0}^m \frac{(N \log j)^{m-l}}{(m - l)!}\right],
\]
and \eqref{3.4.1} follows. The integral in \eqref{3.4.2} is easily evaluated.
\end{proof}

\subsection{A limit calculation}

We will need the following limit in the proof of Theorem \ref{Theorem 1.1}.

\begin{proposition} \label{Proposition 2.7}
We have
\[
\lim_{n \to \infty}\, \int_0^1 ne^{- n\, (t - t^{N'})}\, dt = N.
\]
\end{proposition}

\begin{proof}
Let $f_n(t) = ne^{- n\, (t - t^{N'})}$ and set $t_0 = (N')^{-1/(N'-1)}$. For $t \ne t_0$,
\begin{equation} \label{2.13}
f_n(t) = g_n(t) - \frac{d}{dt} \left(\frac{e^{- n\, (t - t^{N'})}}{1 - N'\, t^{N'-1}}\right),
\end{equation}
where
\[
g_n(t) = \frac{N' (N' - 1)\, t^{N'-2}\, e^{- n\, (t - t^{N'})}}{(1 - N'\, t^{N'-1})^2}.
\]
Fix $\delta$ so small that $0 < \delta < t_0 < 1 - \delta < 1$ and write
\begin{equation} \label{2.14}
\int_0^1 f_n(t)\, dt = \int_0^\delta f_n(t)\, dt + \int_\delta^{1 - \delta} f_n(t)\, dt + \int_{1 - \delta}^1 f_n(t)\, dt.
\end{equation}
By \eqref{2.13},
\begin{equation} \label{2.15}
\int_0^\delta f_n(t)\, dt = \int_0^\delta g_n(t)\, dt - \frac{e^{- n\, (\delta - \delta^{N'})}}{1 - N'\, \delta^{N'-1}} + 1.
\end{equation}
For all $t \in (0,\delta)$, $g_n(t) \to 0$ as $n \to \infty$ and $\abs{g_n(t)} \le N' (N' - 1)\, t^{N'-2}/(1 - N'\, \delta^{N'-1})^2$, so $\int_0^\delta g_n(t)\, dt \to 0$ by the dominated convergence theorem. So $\int_0^1 f_n(t)\, dt \to 1$ by \eqref{2.15}. A similar calculation shows that $\int_{1 - \delta}^1 f_n(t)\, dt \to N - 1$. On the other hand, it is easily seen that $\int_\delta^{1 - \delta} f_n(t)\, dt \to 0$. So $\int_0^1 f_n(t)\, dt \to N$ by \eqref{2.14}.
\end{proof}

\section{Proof of Theorem \ref{Theorem 1.1}}

In this section we prove Theorem \ref{Theorem 1.1} by showing that the functional $E$ has the mountain pass geometry with the mountain pass level $c \in (0,(1/N)(\alpha_N/\alpha)^{N-1})$ and applying Proposition \ref{Proposition 2.1}.

\begin{lemma} \label{Lemma 3.1}
There exists a $\rho > 0$ such that
\[
\inf_{\norm{u} = \rho}\, E(u) > 0.
\]
\end{lemma}

\begin{proof}
Since \eqref{1.1.1} implies that $h$ is bounded, there exists a constant $C_\delta > 0$ such that
\[
|G(t)| \le C_\delta\, |t|^{N+1}\, e^{\alpha\, |t|^{N'}} \quad \text{for } |t| > \delta,
\]
which together with \eqref{1.7} gives
\begin{equation} \label{3.1}
\int_\Omega G(u)\, dx \le \frac{1}{N}\, (\lambda_1 - \sigma_1) \int_\Omega |u|^N dx + C_\delta \int_\Omega |u|^{N+1}\, e^{\alpha\, |u|^{N'}} dx.
\end{equation}
By \eqref{1.5},
\begin{equation}
\int_\Omega |u|^N dx \le \frac{\rho^N}{\lambda_1},
\end{equation}
where $\rho = \norm{u}$. By the H\"{o}lder inequality,
\begin{equation} \label{3.3}
\int_\Omega |u|^{N+1}\, e^{\alpha\, |u|^{N'}} dx \le \left(\int_\Omega |u|^{2\, (N+1)}\, dx\right)^{1/2} \left(\int_\Omega e^{2 \alpha\, |u|^{N'}} dx\right)^{1/2}.
\end{equation}
The first integral on the right-hand side is bounded by $C \rho^{2\, (N+1)}$ for some constant $C > 0$ by the Sobolev embedding theorem. Since $2 \alpha\, |u|^{N'} = 2 \alpha\, \rho^{N'} |\widetilde{u}|^{N'}$, where $\widetilde{u} = u/\rho$ satisfies $\norm{\widetilde{u}} = 1$, the second integral is bounded when $\rho^{N'} \le \alpha_N/2 \alpha$ by \eqref{1.3}. So combining \eqref{3.1}--\eqref{3.3} gives
\[
\int_\Omega G(u)\, dx \le \frac{1}{N} \left(1 - \frac{\sigma_1}{\lambda_1}\right) \rho^N + \O(\rho^{N+1}) \quad \text{as } \rho \to 0.
\]
Then
\[
E(u) \ge \frac{1}{N}\, \frac{\sigma_1}{\lambda_1}\, \rho^N + \O(\rho^{N+1}),
\]
and the desired conclusion follows from this for sufficiently small $\rho > 0$.
\end{proof}

We may assume without loss of generality that $B_d(0) \subset \Omega$. Let $\seq{\omega_j}$ be the sequence of functions defined in \eqref{2.5.1}.

\begin{lemma} \label{Lemma 3.2}
We have
\begin{enumroman}
\item \label{Lemma 3.2 (i)} $E(t \omega_j) \to - \infty$ as $t \to \infty$ for all $j \ge 2$,
\item \label{Lemma 3.2 (ii)} $\exists j_0 \ge 2$ such that
    \[
    \sup_{t \ge 0}\, E(t \omega_{j_0}) < \frac{1}{N} \left(\frac{\alpha_N}{\alpha}\right)^{N-1}.
    \]
\end{enumroman}
\end{lemma}

\begin{proof}
\ref{Lemma 3.2 (i)} Fix $0 < \eps < \beta$. By \eqref{1.2}, $\exists M_\eps > 0$ such that
\begin{equation} \label{3.6}
th(t)\, e^{\alpha\, |t|^{N'}} > (\beta - \eps)\, e^{\alpha\, |t|^{N'}} \quad \text{for } |t| > M_\eps.
\end{equation}
Since $e^{\alpha\, |t|^{N'}} > \alpha^{2N-2}\, t^{2N}/(2N - 2)!$ for all $t$, then there exists a constant $C_\eps > 0$ such that
\begin{equation} \label{3.7}
th(t)\, e^{\alpha\, |t|^{N'}} \ge \frac{1}{(2N - 2)!}\, (\beta - \eps)\, \alpha^{2N-2}\, t^{2N} - C_\eps\, |t|
\end{equation}
and
\begin{equation} \label{3.8}
G(t) \ge \frac{2N - 1}{(2N)!}\, (\beta - \eps)\, \alpha^{2N-2}\, t^{2N} - C_\eps\, |t|
\end{equation}
for all $t$. Since $\norm{\omega_j} = 1$ and $\omega_j \ge 0$, then
\[
E(t \omega_j) \le \frac{t^N}{N} - \frac{2N - 1}{(2N)!}\, (\beta - \eps)\, \alpha^{2N-2}\, t^{2N} \int_\Omega \omega_j^{2N}\, dx + C_\eps\, t \int_\Omega \omega_j\, dx,
\]
and the conclusion follows.

\ref{Lemma 3.2 (ii)} Set
\[
H_j(t) = E(t \omega_j) = \frac{t^N}{N} - \int_\Omega G(t \omega_j)\, dx, \quad t \ge 0.
\]
If the conclusion is false, then it follows from \ref{Lemma 3.2 (i)} that for all $j \ge 2$, $\exists t_j > 0$ such that
\begin{gather}
\label{3.9} H_j(t_j) = \frac{t_j^N}{N} - \int_\Omega G(t_j \omega_j)\, dx = \sup_{t \ge 0}\, H_j(t) \ge \frac{1}{N} \left(\frac{\alpha_N}{\alpha}\right)^{N-1},\\[7.5pt]
\label{3.10} H_j'(t_j) = t_j^{N-1} - \int_\Omega \omega_j\, h(t_j \omega_j)\, e^{\alpha\, t_j^{N'} \omega_j^{N'}} dx = 0.
\end{gather}
Since $G(t) \ge - C_\eps\, t$ for all $t \ge 0$ by \eqref{3.8}, \eqref{3.9} gives
\begin{equation} \label{3.11}
t_j^N \ge t_0^N - N \delta_j\, t_j,
\end{equation}
where
\[
t_0 = \left(\frac{\alpha_N}{\alpha}\right)^{(N-1)/N}
\]
and
\begin{equation} \label{3.110}
\delta_j = C_\eps \int_\Omega \omega_j\, dx \to 0 \quad \text{as } j \to \infty
\end{equation}
by Proposition \ref{Proposition 2.6}. First we will show that $t_j \to t_0$.

By \eqref{3.11} and the Young's inequality,
\[
(1 + \nu)\, t_j^N \ge t_0^N - \frac{N - 1}{\nu^{1/(N-1)}}\, \delta_j^{N'} \quad \forall \nu > 0,
\]
which together with \eqref{3.110} gives
\begin{equation} \label{3.13}
\liminf_{j \to \infty}\, t_j \ge t_0.
\end{equation}
Write \eqref{3.10} as
\begin{equation} \label{3.14}
t_j^N = \int_{\set{t_j \omega_j > M_\eps}} t_j \omega_j\, h(t_j \omega_j)\, e^{\alpha\, t_j^{N'} \omega_j^{N'}} dx + \int_{\set{t_j \omega_j \le M_\eps}} t_j \omega_j\, h(t_j \omega_j)\, e^{\alpha\, t_j^{N'} \omega_j^{N'}} dx =: I_1 + I_2.
\end{equation}
Set $r_j = de^{- M_\eps\, (\omega_{N-1} \log j)^{1/N}/t_j}$. Since $\liminf t_j > 0$, for all sufficiently large $j$, $d/j < r_j < d$ and $t_j \omega_j(x) > M_\eps$ if and only if $|x| < r_j$. So \eqref{3.6} gives
\begin{multline}
I_1 \ge (\beta - \eps) \int_{\set{|x| < r_j}} e^{\alpha\, t_j^{N'} \omega_j^{N'}} dx = (\beta - \eps) \Bigg(\int_{\set{|x| \le d/j}} e^{\alpha\, t_j^{N'} \omega_j^{N'}} dx\\[7.5pt]
+ \int_{\set{d/j < |x| < r_j}} e^{\alpha\, t_j^{N'} \omega_j^{N'}} dx\Bigg) =: (\beta - \eps)\, (I_3 + I_4).
\end{multline}
We have
\begin{equation}
I_3 = \frac{\omega_{N-1}}{N} \left(\frac{d}{j}\right)^N e^{\alpha\, t_j^{N'} \log j/\omega_{N-1}^{1/(N-1)}} = \frac{\omega_{N-1}}{N}\, d^N j^{\alpha\, (t_j^{N'} - t_0^{N'})/\omega_{N-1}^{1/(N-1)}}.
\end{equation}
Since $th(t)\, e^{\alpha\, |t|^{N'}} \ge - C_\eps\, t$ for all $t \ge 0$ by \eqref{3.7},
\begin{equation} \label{3.17}
I_2 \ge - C_\eps\, t_j \int_{\set{t_j \omega_j \le M_\eps}} \omega_j\, dx \ge - \delta_j\, t_j.
\end{equation}
Combining \eqref{3.14}--\eqref{3.17} and noting that $I_4 \ge 0$ gives
\[
t_j^N \ge (\beta - \eps)\, \frac{\omega_{N-1}}{N}\, d^N j^{\alpha\, (t_j^{N'} - t_0^{N'})/\omega_{N-1}^{1/(N-1)}} - \delta_j\, t_j.
\]
It follows from this that
\[
\limsup_{j \to \infty}\, t_j \le t_0,
\]
which together with \eqref{3.13} shows that $t_j \to t_0$.

Next we estimate $I_4$. We have
\begin{align}
I_4 & = \int_{\set{d/j < |x| < r_j}} e^{\alpha\, t_j^{N'} [\log\, (d/|x|)]^{N'}/(\omega_{N-1} \log j)^{1/(N-1)}}\, dx \notag\\[7.5pt]
& = \omega_{N-1}\, \Bigg(\int_{d/j}^d e^{\alpha\, t_j^{N'} [\log\, (d/r)]^{N'}/(\omega_{N-1} \log j)^{1/(N-1)}}\, r^{N-1}\, dr \notag\\[7.5pt]
& \phantom{=} - \int_{r_j}^d e^{\alpha\, t_j^{N'} [\log\, (d/r)]^{N'}/(\omega_{N-1} \log j)^{1/(N-1)}}\, r^{N-1}\, dr\Bigg) \notag\\[7.5pt]
& = \omega_{N-1}\, d^N \Bigg(\log j \int_0^1 e^{- Nt\, [1 - (t_j/t_0)^{N'} t^{1/(N-1)}] \log j}\, dt \notag\\[7.5pt]
& \phantom{=} - \int_{s_j}^1 s^{N-1}\, e^{\alpha\, t_j^{N'} (- \log s)^{N'}/(\omega_{N-1} \log j)^{1/(N-1)}}\, ds\Bigg), \label{3.19}
\end{align}
where $t = \log\, (d/r)/\log j$, $s = r/d$, and $s_j = r_j/d = e^{- M_\eps\, (\omega_{N-1} \log j)^{1/N}/t_j} \to 0$. For $s_j < s < 1$, $\alpha\, t_j^{N'} (- \log s)^{N'}/(\omega_{N-1} \log j)^{1/(N-1)}$ is bounded by $\alpha M_\eps^{N'}$ and goes to zero as $j \to \infty$, so the last integral converges to
\[
\int_0^1 s^{N-1}\, ds = \frac{1}{N}.
\]
So combining \eqref{3.14}--\eqref{3.19} and letting $j \to \infty$ gives
\[
t_0^N \ge (\beta - \eps)\, \frac{\omega_{N-1}}{N}\, d^N (L_1 + L_2 - 1),
\]
where
\begin{gather*}
L_1 = \liminf_{j \to \infty}\, e^{- n\, [1 - (t_j/t_0)^{N'}]},\\[7.5pt]
L_2 = \liminf_{j \to \infty}\, \int_0^1 ne^{- n\, [t - (t_j/t_0)^{N'} t^{N'}]}\, dt,
\end{gather*}
and $n = N \log j \to \infty$. Letting $\eps \to 0$ in this inequality gives
\begin{equation} \label{3.20}
\beta \le \frac{1}{\alpha^{N-1}} \left(\frac{N}{d}\right)^N \frac{1}{L_1 + L_2 - 1}.
\end{equation}

By \eqref{3.9}, \eqref{1.9}, and Proposition \ref{Proposition 2.6},
\[
t_j^N - t_0^N \ge N \int_\Omega G(t_j \omega_j)\, dx \ge - \sigma_0\, t_j^N \int_\Omega \omega_j^N\, dx \ge - \frac{\sigma_0\, t_j^N}{\kappa n},
\]
so
\[
\left(\frac{t_j}{t_0}\right)^{N'} \ge \left(1 + \frac{\sigma_0}{\kappa n}\right)^{- 1/(N-1)} \ge 1 - \frac{\sigma_0}{(N - 1)\, \kappa n}.
\]
This gives
\[
L_1 \ge e^{- \sigma_0/(N-1)\, \kappa}
\]
and
\[
L_2 \ge \lim_{n \to \infty}\, \int_0^1 ne^{- n\, (t - t^{N'}) - \sigma_0\, t^{N'}/(N-1)\, \kappa}\, dt \ge Ne^{- \sigma_0/(N-1)\, \kappa}
\]
by Proposition \ref{Proposition 2.7}. So \eqref{3.20} gives
\[
\beta \le \frac{1}{\alpha^{N-1}} \left(\frac{N}{d}\right)^N \frac{1}{Ne^{- \sigma_0/(N-1)\, \kappa} - (1 - e^{- \sigma_0/(N-1)\, \kappa})} \le \frac{1}{N \alpha^{N-1}} \left(\frac{N}{d}\right)^N\! e^{\sigma_0/(N-1)\, \kappa},
\]
contradicting \eqref{1.10}.
\end{proof}

We are now ready to prove Theorem \ref{Theorem 1.1}.

\begin{proof}[Proof of Theorem \ref{Theorem 1.1}]
Let $j_0$ be as in Lemma \ref{Lemma 3.2} \ref{Lemma 3.2 (ii)}. By Lemma \ref{Lemma 3.2} \ref{Lemma 3.2 (i)}, $\exists R > \rho$ such that $E(R \omega_{j_0}) \le 0$, where $\rho$ is as in Lemma \ref{Lemma 3.1}. Let
\[
\Gamma = \set{\gamma \in C([0,1],W^{1,N}_0(\Omega)) : \gamma(0) = 0,\, \gamma(1) = R \omega_{j_0}}
\]
be the class of paths joining the origin to $R \omega_{j_0}$, and set
\[
c := \inf_{\gamma \in \Gamma}\, \max_{u \in \gamma([0,1])}\, E(u).
\]
By Lemma \ref{Lemma 3.1}, $c > 0$. Since the path $\gamma_0(t) = tR \omega_{j_0},\, t \in [0,1]$ is in $\Gamma$,
\[
c \le \max_{u \in \gamma_0([0,1])}\, E(u) \le \sup_{t \ge 0}\, E(t \omega_{j_0}) < \frac{1}{N} \left(\frac{\alpha_N}{\alpha}\right)^{N-1}.
\]
If there are no \PS{c} sequences of $E$, then $E$ satisfies the \PS{c} condition vacuously and hence has a critical point $u$ at the level $c$ by the mountain pass theorem. Then $u$ is a solution of problem \eqref{1.1} and $u$ is nontrivial since $c > 0$. So we may assume that $E$ has a \PS{c} sequence. Then this sequence has a subsequence that converges weakly to a nontrivial solution of problem \eqref{1.1} by Proposition \ref{Proposition 2.1}.
\end{proof}

\section{Proof of Theorem \ref{Theorem 1.3}}

In this section we prove Theorem \ref{Theorem 1.3} using Theorem \ref{Theorem 2.2}. We take $A_0$ to be the set $C$ in Proposition \ref{Proposition 2.4} and $B_0 = \Psi_{\lambda_k}$. Since $i(S \setminus B_0) = k - 1$ by Proposition \ref{Proposition 2.3}, \eqref{2.9.1} holds.

\begin{lemma} \label{Lemma 4.1}
There exists a $\rho > 0$ such that $\inf E(B) > 0$, where $B = \set{\rho u : u \in B_0}$.
\end{lemma}

\begin{proof}
As in the proof of Lemma \ref{Lemma 3.1}, there exists a constant $C_\delta > 0$ such that
\[
|G(t)| \le C_\delta\, |t|^{N+1}\, e^{\alpha\, |t|^{N'}} \quad \text{for } |t| > \delta,
\]
which together with \eqref{1.12} gives
\begin{equation} \label{4.1}
G(t) \le \frac{1}{N}\, (\lambda_k - \sigma_1)\, |t|^N + C_\delta\, |t|^{N+1}\, e^{\alpha\, |t|^{N'}} \quad \forall t.
\end{equation}
For $u \in B_0$ and $\rho > 0$,
\begin{equation}
\int_\Omega |\rho u|^N dx \le \frac{\rho^N}{\lambda_k}
\end{equation}
and
\begin{equation} \label{4.3}
\int_\Omega |\rho u|^{N+1}\, e^{\alpha\, |\rho u|^{N'}} dx \le \rho^{N+1} \left(\int_\Omega |u|^{2\, (N+1)}\, dx\right)^{1/2} \left(\int_\Omega e^{2 \alpha\, \rho^{N'} |u|^{N'}} dx\right)^{1/2}.
\end{equation}
The first integral on the right-hand side of \eqref{4.3} is bounded by the Sobolev embedding theorem, and the second integral is bounded when $\rho^{N'} \le \alpha_N/2 \alpha$ by \eqref{1.3}. So combining \eqref{4.1}--\eqref{4.3} gives
\[
\int_\Omega G(\rho u)\, dx \le \frac{1}{N} \left(1 - \frac{\sigma_1}{\lambda_k}\right) \rho^N + \O(\rho^{N+1}) \quad \text{as } \rho \to 0.
\]
Then
\[
E(\rho u) \ge \frac{1}{N}\, \frac{\sigma_1}{\lambda_k}\, \rho^N + \O(\rho^{N+1}),
\]
and the desired conclusion follows from this for sufficiently small $\rho$.
\end{proof}

We may assume without loss of generality that $B_d(0) \subset \Omega$. Let $\seq{\omega_j}$ be the sequence of functions defined in \eqref{2.5.1}.

\begin{lemma} \label{Lemma 4.2}
We have
\begin{enumroman}
\item \label{Lemma 4.2 (i)} $E(sv) \le 0 \hquad \forall v \in A_0,\, s \ge 0$,
\item \label{Lemma 4.2 (ii)} for all $j \ge 2$,
    \[
    \sup \set{E(R\, \pi((1 - t)\, v + t \omega_j)) : v \in A_0,\, 0 \le t \le 1} \to - \infty$ as $R \to \infty,
    \]
\item \label{Lemma 4.2 (iii)} $\exists j_0 \ge 2$ such that
    \[
    \sup \set{E(sv + t \omega_{j_0}) : v \in A_0,\, s, t \ge 0} < \frac{1}{N} \left(\frac{\alpha_N}{\alpha}\right)^{N-1}.
    \]
\end{enumroman}
\end{lemma}

\begin{proof}
\ref{Lemma 4.2 (i)} By \eqref{1.11},
\begin{equation} \label{4.4}
E(u) \le \frac{1}{N} \left[\int_\Omega |\nabla u|^N dx - (\lambda_{k-1} + \sigma_0) \int_\Omega |u|^N dx\right].
\end{equation}
For $v \in A_0$ and $s \ge 0$,
\[
\int_\Omega |sv|^N dx \ge \frac{s^N}{\lambda_{k-1}}
\]
since $A_0 \subset \Psi^{\lambda_{k-1}}$, so \eqref{4.4} gives
\[
E(sv) \le - \frac{1}{N}\, \frac{\sigma_0}{\lambda_{k-1}}\, s^N \le 0.
\]

\ref{Lemma 4.2 (ii)} Fix $0 < \eps < \beta$. As in the proof of Lemma \ref{Lemma 3.2} \ref{Lemma 3.2 (i)}, $\exists M_\eps > 0$ such that
\begin{equation} \label{3.16}
th(t)\, e^{\alpha\, |t|^{N'}} > (\beta - \eps)\, e^{\alpha\, |t|^{N'}} \quad \text{for } |t| > M_\eps
\end{equation}
and there exists a constant $C_\eps > 0$ such that
\begin{equation} \label{3.21}
th(t)\, e^{\alpha\, |t|^{N'}} \ge \frac{1}{(2N - 2)!}\, (\beta - \eps)\, \alpha^{2N-2}\, t^{2N} - C_\eps\, |t|
\end{equation}
and
\begin{equation} \label{3.18}
G(t) \ge \frac{2N - 1}{(2N)!}\, (\beta - \eps)\, \alpha^{2N-2}\, t^{2N} - C_\eps\, |t|
\end{equation}
for all $t$. Let $A_1 = \set{\pi((1 - t)\, v + t \omega_j) : v \in A_0,\, 0 \le t \le 1}$. For $u \in A_1$ and $R > 0$, \eqref{3.18} gives
\[
E(Ru) \le \frac{R^N}{N} - \frac{2N - 1}{(2N)!}\, (\beta - \eps)\, \alpha^{2N-2}\, R^{2N} \int_\Omega |u|^{2N} dx + C_\eps R \int_\Omega |u|\, dx.
\]
The set $A_1$ is compact since $A_0$ is compact, so the first integral on the right-hand side is bounded away from zero on $A_1$. Since the second integral is bounded, the desired conclusion follows.

\ref{Lemma 4.2 (iii)} If the conclusion is false, then it follows from \ref{Lemma 4.2 (i)} and \ref{Lemma 4.2 (ii)} that for all $j \ge 2$, there exist $v_j \in A_0,\, s_j \ge 0,\, t_j > 0$ such that
\[
E(s_j v_j + t_j \omega_j) = \sup \set{E(sv + t \omega_j) : v \in A_0,\, s, t \ge 0} \ge \frac{1}{N} \left(\frac{\alpha_N}{\alpha}\right)^{N-1}.
\]
Set $u_j = s_j v_j + t_j \omega_j$. Then
\begin{equation} \label{4.6}
E(u_j) = \frac{1}{N} \norm{u_j}^N - \int_\Omega G(u_j)\, dx \ge \frac{1}{N} \left(\frac{\alpha_N}{\alpha}\right)^{N-1}.
\end{equation}
Moreover, $\tau u_j \in \set{sv + t \omega_j : v \in A_0,\, s, t \ge 0}$ for all $\tau \ge 0$ and $E(\tau u_j)$ attains its maximum at $\tau = 1$, so
\begin{equation} \label{4.7}
\restr{\frac{\partial}{\partial \tau}\, E(\tau u_j)}{\tau = 1} = E'(u_j)\, u_j = \norm{u_j}^N - \int_\Omega u_j\, h(u_j)\, e^{\alpha\, |u_j|^{N'}} dx = 0.
\end{equation}
Since $\norm{v_j} = \norm{\omega_j} = 1$ and $G(t) \ge 0$ for all $t$ by \eqref{1.11}, \eqref{4.6} gives
\[
s_j + t_j \ge t_0,
\]
where
\[
t_0 = \left(\frac{\alpha_N}{\alpha}\right)^{(N-1)/N}.
\]
First we show that $s_j \to 0$ and $t_j \to t_0$ as $j \to \infty$.

Combining \eqref{4.6} with \eqref{1.11} gives
\[
\norm{s_j v_j + t_j \omega_j}^N \ge (\lambda_{k-1} + \sigma_0) \int_\Omega |s_j v_j + t_j \omega_j|^N dx + t_0^N.
\]
Set $\tau_j = s_j/t_j$. Then
\begin{equation} \label{4.11}
\norm{\tau_j v_j + \omega_j}^N \ge (\lambda_{k-1} + \sigma_0) \int_\Omega |\tau_j v_j + \omega_j|^N dx + \left(\frac{t_0}{t_j}\right)^N.
\end{equation}
Since $\seq{v_j}$ is bounded in $C^1(\closure{\Omega})$, Proposition \ref{Proposition 2.6} gives
\begin{multline*}
\norm{\tau_j v_j + \omega_j}^N \le \int_\Omega (\tau_j\, |\nabla v_j| + |\nabla \omega_j|)^N dx = \tau_j^N \int_\Omega |\nabla v_j|^N dx + \int_\Omega |\nabla \omega_j|^N dx\\[7.5pt]
+ \sum_{m=1}^{N-1} \binom{N}{m}\, \tau_j^{N-m} \int_\Omega |\nabla v_j|^{N-m}\, |\nabla \omega_j|^m\, dx \le \tau_j^N + 1 + c_1 \sum_{m=1}^{N-1} \frac{\tau_j^{N-m}}{(\log j)^{m/N}}
\end{multline*}
and
\begin{multline*}
\int_\Omega |\tau_j v_j + \omega_j|^N dx \ge \int_\Omega (\tau_j\, |v_j| - \omega_j)^N dx = \tau_j^N \int_\Omega |v_j|^N dx\\[7.5pt]
+ \sum_{m=1}^N (-1)^m\, \binom{N}{m}\, \tau_j^{N-m} \int_\Omega |v_j|^{N-m}\, \omega_j^m\, dx \ge \frac{\tau_j^N}{\lambda_{k-1}} - c_2 \sum_{m=1}^N \frac{\tau_j^{N-m}}{(\log j)^{m/N}}
\end{multline*}
for some constants $c_1, c_2 > 0$. So \eqref{4.11} gives
\begin{equation} \label{4.12}
\frac{\sigma_0}{\lambda_{k-1}}\, \tau_j^N + \left(\frac{t_0}{t_j}\right)^N \le 1 + c_3 \sum_{m=1}^N \frac{\tau_j^{N-m}}{(\log j)^{m/N}}
\end{equation}
for some constant $c_3 > 0$, which implies that $\seq{\tau_j}$ is bounded and
\begin{equation} \label{4.9.3}
\liminf_{j \to \infty}\, t_j \ge t_0.
\end{equation}

Next combining \eqref{4.7} with \eqref{3.16} and \eqref{3.21} gives
\begin{multline} \label{4.10}
\norm{u_j}^N = \int_{\set{|u_j| > M_\eps}} u_j\, h(u_j)\, e^{\alpha\, |u_j|^{N'}} dx + \int_{\set{|u_j| \le M_\eps}} u_j\, h(u_j)\, e^{\alpha\, |u_j|^{N'}} dx\\[7.5pt]
\ge (\beta - \eps) \int_{\set{|u_j| > M_\eps}} e^{\alpha\, |u_j|^{N'}} dx - C_\eps \int_{\set{|u_j| \le M_\eps}} |u_j|\, dx.
\end{multline}
For $|x| \le d/j$,
\[
|u_j| \ge t_j \omega_j - s_j\, |v_j| \ge \frac{t_j}{\omega_{N-1}^{1/N}} \left[(\log j)^{(N-1)/N} - c_4 \tau_j\right]
\]
for some constant $c_4 > 0$, and the last expression is greater than $M_\eps$ for all sufficiently large $j$ since $\seq{\tau_j}$ is bounded and $\liminf t_j > 0$. So
\begin{multline*}
\int_{\set{|u_j| > M_\eps}} e^{\alpha\, |u_j|^{N'}} dx \ge e^{\alpha\, t_j^{N'} [(\log j)^{(N-1)/N} - c_4 \tau_j]^{N'}/\omega_{N-1}^{1/(N-1)}} \int_{\set{|x| \le d/j}} dx\\[7.5pt]
= \frac{\omega_{N-1}\, d^N}{N}\, j^{\alpha\, [t_j^{N'} (1 - c_4 \tau_j/(\log j)^{(N-1)/N})^{N'} - t_0^{N'}]/\omega_{N-1}^{1/(N-1)}}
\end{multline*}
for large $j$. On the other hand,
\[
\int_{\set{|u_j| \le M_\eps}} |u_j|\, dx \le \int_\Omega (s_j\, |v_j| + t_j \omega_j)\, dx \le c_5\, t_j \left[\tau_j + \frac{1}{(\log j)^{1/N}}\right]
\]
for some constant $c_5 > 0$ by Proposition \ref{Proposition 2.6}. So \eqref{4.10} gives
\begin{multline} \label{4.15}
(\beta - \eps)\, j^{\alpha\, [t_j^{N'} (1 - c_4 \tau_j/(\log j)^{(N-1)/N})^{N'} - t_0^{N'}]/\omega_{N-1}^{1/(N-1)}} \le \frac{N t_j^N (\tau_j + 1)^N}{\omega_{N-1}\, d^N}\\[7.5pt]
+ c_6\, t_j \left[\tau_j + \frac{1}{(\log j)^{1/N}}\right]
\end{multline}
for some constant $c_6 > 0$. Since $\seq{\tau_j}$ is bounded, it follows from this that
\[
\limsup_{j \to \infty}\, t_j \le t_0,
\]
which together with \eqref{4.9.3} shows that $t_j \to t_0$. Then \eqref{4.12} implies that $\tau_j \to 0$, so $s_j = \tau_j\, t_j \to 0$.

Now we show that there exists a constant $c > 0$ depending only on $\Omega$, $\alpha$, and $k$ such that
\begin{equation} \label{4.16}
\beta \le \frac{1}{\alpha^{N-1}} \left(\frac{N}{d}\right)^N\! e^{c/\sigma_0^{N-1}}.
\end{equation}
The right-hand side of \eqref{4.15} goes to $(N/d)^N/\alpha^{N-1}$ as $j \to \infty$. If $\beta \le (N/d)^N/\alpha^{N-1}$, then we may take any $c > 0$, so suppose $\beta > (N/d)^N/\alpha^{N-1}$. Then for $\eps < \beta - (N/d)^N/\alpha^{N-1}$ and all sufficiently large $j$, \eqref{4.15} gives $j^{\alpha\, [t_j^{N'} (1 - c_4 \tau_j/(\log j)^{(N-1)/N})^{N'} - t_0^{N'}]/\omega_{N-1}^{1/(N-1)}} \le 1$, so
\[
\frac{t_0}{t_j} \ge 1 - \frac{c_4 \tau_j}{(\log j)^{(N-1)/N}}.
\]
Combining this with \eqref{4.12} gives
\[
\frac{\sigma_0}{\lambda_{k-1}}\, \tau_j^N - \frac{N c_4 \tau_j}{(\log j)^{(N-1)/N}} \le c_3 \sum_{m=1}^N \frac{\tau_j^{N-m}}{(\log j)^{m/N}},
\]
so
\[
\sigma_0 \tau_j^N \le c_7 \sum_{m=1}^N \frac{\tau_j^{N-m}}{(\log j)^{m/N}}
\]
for some constant $c_7 > 0$. Set $\widetilde{\tau}_j = \tau_j\, (\log j)^{1/N}$. Then
\begin{equation} \label{4.18}
\sigma_0 \widetilde{\tau}_j^N \le c_7 \sum_{m=1}^N \widetilde{\tau}_j^{N-m}.
\end{equation}
We claim that
\begin{equation} \label{4.17}
\widetilde{\tau}_j \le \frac{c_8}{\sigma_0}
\end{equation}
for some constant $c_8 > 0$. Taking $\sigma_0$ smaller in \eqref{1.11} if necessary, we may assume that $\sigma_0 \le 1$. So if $\widetilde{\tau}_j < 1$, then \eqref{4.17} holds with $c_8 = 1$, so suppose $\widetilde{\tau}_j \ge 1$. Then \eqref{4.18} gives \eqref{4.17} with $c_8 = N c_7$. Now \eqref{4.12} gives
\[
\left(\frac{t_0}{t_j}\right)^N \le 1 + \frac{c_3}{\log j}\, \sum_{m=1}^N \widetilde{\tau}_j^{N-m} \le 1 + \frac{c_9}{\sigma_0^{N-1} \log j}
\]
for some constant $c_9 > 0$, so
\[
\left(\frac{t_0}{t_j}\right)^{N'} \le \left(1 + \frac{c_9}{\sigma_0^{N-1} \log j}\right)^{1/(N-1)} \le 1 + \frac{c_9}{\sigma_0^{N-1} \log j}.
\]
Then
\begin{multline*}
t_j^{N'} \left[1 - \frac{c_4 \tau_j}{(\log j)^{(N-1)/N}}\right]^{N'} - t_0^{N'} = t_j^{N'} \left[\left(1 - \frac{c_4 \widetilde{\tau}_j}{\log j}\right)^{N'} - \left(\frac{t_0}{t_j}\right)^{N'}\right]\\[7.5pt]
\ge t_j^{N'} \left[\left(1 - \frac{c_{10}}{\sigma_0 \log j}\right)^{N'} - \left(1 + \frac{c_9}{\sigma_0^{N-1} \log j}\right)\right] \ge - t_j^{N'} \left(\frac{N' c_{10}}{\sigma_0 \log j} + \frac{c_9}{\sigma_0^{N-1} \log j}\right)\\[7.5pt]
\ge - \frac{c_{11}}{\sigma_0^{N-1} \log j}
\end{multline*}
for some constants $c_{10}, c_{11} > 0$, so
\[
j^{\alpha\, [t_j^{N'} (1 - c_4 \tau_j/(\log j)^{(N-1)/N})^{N'} - t_0^{N'}]/\omega_{N-1}^{1/(N-1)}} \ge j^{- c/\sigma_0^{N-1} \log j} = e^{- c/\sigma_0^{N-1}}
\]
for some constant $c > 0$. Combining this with \eqref{4.15} and passing to the limit gives
\[
(\beta - \eps)\, e^{- c/\sigma_0^{N-1}} \le \frac{1}{\alpha^{N-1}} \left(\frac{N}{d}\right)^N,
\]
and letting $\eps \to 0$ gives \eqref{4.16}.
\end{proof}

We are now ready to prove Theorem \ref{Theorem 1.3}.

\begin{proof}[Proof of Theorem \ref{Theorem 1.3}]
Let $j_0 \ge 2$ be as in Lemma \ref{Lemma 4.2} \ref{Lemma 4.2 (iii)}. By Lemma \ref{Lemma 4.2} \ref{Lemma 4.2 (ii)}, $\exists R > \rho$ such that
\begin{equation} \label{4.19}
\sup \set{E(R\, \pi((1 - t)\, v + t \omega_{j_0})) : v \in A_0,\, 0 \le t \le 1} \le 0,
\end{equation}
where $\rho > 0$ is as in Lemma \ref{Lemma 4.1}. Let
\begin{gather*}
A = \set{sv : v \in A_0,\, 0 \le s \le R} \cup \set{R\, \pi((1 - t)\, v + t \omega_{j_0}) : v \in A_0,\, 0 \le t \le 1},\\[7.5pt]
X = \set{sv + t \omega_{j_0} : v \in A_0,\, s, t \ge 0,\, \norm{sv + t \omega_{j_0}} \le R}.
\end{gather*}
Combining Lemma \ref{Lemma 4.2} \ref{Lemma 4.2 (i)}, \eqref{4.19}, and Lemma \ref{Lemma 4.1} gives
\begin{equation} \label{4.20}
\sup E(A) \le 0 < \inf E(B),
\end{equation}
while Lemma \ref{Lemma 4.2} \ref{Lemma 4.2 (iii)} gives
\begin{equation} \label{4.21}
\sup E(X) \le \sup \set{E(sv + t \omega_{j_0}) : v \in A_0,\, s, t \ge 0} < \frac{1}{N} \left(\frac{\alpha_N}{\alpha}\right)^{N-1}.
\end{equation}
Let
\[
\Gamma = \set{\gamma \in C(X,W) : \gamma(X) \text{ is closed and} \restr{\gamma}{A} = \id{A}},
\]
and set
\[
c := \inf_{\gamma \in \Gamma}\, \sup_{u \in \gamma(X)}\, E(u).
\]
By Theorem \ref{Theorem 2.2}, $\inf E(B) \le c \le \sup E(X)$, and $E$ has a \PS{c} sequence. By \eqref{4.20} and \eqref{4.21},
\[
0 < c < \frac{1}{N} \left(\frac{\alpha_N}{\alpha}\right)^{N-1},
\]
so a subsequence of this \PS{c} sequence converges weakly to a nontrivial solution of problem \eqref{1.1} by Proposition \ref{Proposition 2.1}.
\end{proof}

\section{Competing interests declaration}

The authors declare no competing interests.

\def\cdprime{$''$}

\end{document}